\documentclass{amsart}
\usepackage{amssymb, latexsym, amsmath, amsthm, amsfonts}

\DeclareMathOperator{\Aut}{Aut}

\DeclareMathOperator{\Prob}{Prob}

\newcommand{\draftcom}[1]{}  

\newtheorem{theorem}{Theorem}[section]

\newtheorem{proposition}[theorem]{Proposition}
\newtheorem{corollary}[theorem]{Corollary}
\newtheorem{lemma}[theorem]{Lemma}

\newtheorem{remark}[theorem]{Remark}

\def \F {{\mathbb F}}

\def \P {{\mathbb P}}

\title[Note On Plane Curves]
{The fluctuations in the number of points of smooth plane curves over finite fields}
\author{Alina Bucur,
Chantal David,
Brooke Feigon,
Matilde Lal\'in }

\address{Institute for Advanced Study and University of California at San Diego} \email{alina@math.ias.edu}
\address{Concordia University and Institute for Advanced Study} \email{cdavid@mathstat.concordia.ca}
 \address{University of Toronto} \email{bfeigon@math.toronto.edu}
 \address{University of Alberta} \email{mlalin@math.ualberta.ca}
\begin{document}

\maketitle
\baselineskip 14pt

\section{Introduction}

In this note, we study the fluctuations in the number of points of
smooth projective plane curves over finite fields $\F_q$ as $q$ is
fixed and the genus varies. More precisely, we show that these fluctuations are predicted by a natural probabilistic model, in which the points of the projective plane impose independent conditions on the curve. The main tool we use is a geometric sieving process introduced by Poonen \cite{p}.

Let $S_d$ be the set of homogeneous polynomials $F(X,Y,Z)$ of degree
$d$ over $\F_q$, and let $S_d^{\rm ns} \subseteq S_d$ be the subset
of polynomials corresponding to smooth (or non-singular) curves
$C_F: F(X,Y,Z)=0$. The genus of $C_F$ is $(d-1)(d-2)/2.$ By running
over all polynomials $F \in S_d^{\text ns}$, one would expect
the average number of points of $C_F(\F_q)$ to be $q+1$. We show that
this is true, and that the difference between $\#C_F(\F_q)$ and
$q+1$ (properly normalized) tends to a standard Gaussian when $q$
and $d$ tend to infinity. As indicated before, our main tool is a sieving
process due to Poonen \cite{p} which allows us to count the number
of polynomials in $S_d$ which give rise to smooth curves $C_F$, and
the number of smooth curves $C_F$ which pass through a fixed set of
points of $\P^2(\F_q)$. We denote by $p$ the characteristic of $\mathbb F_q.$

\begin{theorem}\label{thm1} Let $X_1, \dots, X_{q^2+q+1}$ be $q^2+q+1$ i.i.d. random variables taking the
value $1$ with probability $(q+1)/(q^2+q+1)$ and the value $0$ with probability $q^2/(q^2+q+1).$
Then
\begin{eqnarray*}
\frac{\# \left\{ F \in {S}_d^{\rm ns} ; \#C_F(\F_q) = t \right\}}{\# {S}_d^{\rm ns}} &=&
\Prob\left( X_1 + \dots + X_{q^2+q+1} = t \right) \\
&\times& \left(1+ O\left(q^{t}\left({d^{-1/3}}+ (d-1)^2 q^{-\min\left(\left\lfloor \frac{d}{p}\right\rfloor+1, \frac{d}{3} \right)} + d q^{-\left\lfloor\frac{d-1}{p}\right\rfloor-1}\right)\right) \right).
 \end{eqnarray*}
\end{theorem}

We now explain why these random variables model the point count for smooth curves. Intuitively, if $F$ is any polynomial in $S_d$, then the set of $\mathbb F_q$-points of the curve $C_F$ is a subset of $\mathbb P^2(\mathbb F_q)$, which has
$q^2+q+1$ elements. Heuristically, these points impose independent
conditions on $F$.

Let us look at one of those conditions, say at the point $[0:0:1].$ Put
$f(x,y) = F(X,Y,1)$  and write
\[f(x,y) = a_{0,0} + a_{1,0} x + a_{0,1} y + \ldots. \]

Since we insist that $C_F$ is smooth, we cannot have $(a_{0,0}, a_{1,0},
a_{0,1}) = (0,0,0),$ so there are $q^3 - 1$ possibilities for this triple.
Of these, the ones that actually correspond to cases where $[0:0:1]$ is on the curve $C_F$
are those where $a_{0,0} = 0,$ of which there are $q^2 - 1.$ So the
probability that $[0:0:1]$ lies on $C_F$ is
  \[ \frac{q^2 - 1}{q^3 - 1} = \frac{q + 1}{q^2 + q + 1}.\]
The argument works the same for any point in the plane, and in particular the expected number of points in $C_F(\mathbb F_q)$ is $q+1.$ This explains the random variables of Theorem \ref{thm1}.
Namely, the probability that $X=1$ (respectively $X=0$)
is the probability that a point $P \in \P^2(\F_q)$ belongs (respectively does not belong)
to a smooth curve $F(X,Y,Z)=0$.

\begin{remark}
We may also let $q$ go to infinity in Theorem \ref{thm1} provided that $d$ goes to infinity in such a way that $d>q^{3(q^2+q+1)+\varepsilon}.$ By studying the moments we can substantially weaken this condition.
\end{remark}

Since the average value of each of the random variables $X_i$ is $(q+1)/(q^2+q+1)$, and the
standard deviation is $q \sqrt{q+1}/(q^2+q+1)$, it follows from the Triangular Central Limit Theorem
\cite{bi}
that
$$
\frac{(X_1+\ldots+X_{q^2+q+1}) - (q+1)}{\sqrt{q+1}} \rightarrow N(0,1)
$$
as $q$ tends to infinity. We can show that this also holds for $\# C_F(\F_q)$ for $F \in {S}_d^{\rm ns},$
as $q$ and $d$ tend to infinity with $d > q^{1+\varepsilon}$, by showing that, under these conditions,  the integral moments of
$$
\frac{\# C_F(\F_q) - (q+1)}{\sqrt{q+1}}
$$
converge to the integral moments of $\displaystyle \frac{(X_1+\ldots+X_{q^2+q+1}) - (q+1)}{\sqrt{q+1}}.$ 

\begin{theorem}\label{thm2}
Let $k$ be a positive integer, and let
$$
M_k(q,d) = \frac{1}{\# {S}_d^{\rm ns}} \sum_{F \in {S}_d^{\rm ns}} \left(
\frac{\# C_F(\F_q) - (q+1)}{\sqrt{q+1}} \right)^k.
$$
Then,
\begin{eqnarray*}
M_k(q,d) &=& {\mathbb{E}} \left( \left( \frac{1}{\sqrt{q+1}} \left( \sum_{i=1}^{q^2+q+1} X_i - (q+1) \right) \right)^k \right) \\
&\times& \left(1+ O\left( q^{\min(k, q^2+q+1)} \left( q^{-k}d^{-1/3} + (d-1)^2 q^{-\min\left(\left\lfloor \frac{d}{p}\right\rfloor+1, \frac{d}{3} \right)} + d q^{-\left\lfloor\frac{d-1}{p}\right\rfloor-1}\right)\right) \right).
\end{eqnarray*}
\end{theorem}

\begin{corollary}
When $q$ and  $d$ tend to infinity and $d > q^{1+\varepsilon}$,
$$
\frac{\# C_F(\F_q) - (q+1)}{\sqrt{q+1}} \rightarrow N(0,1).
$$
\end{corollary}

Finally, we point out that Theorem \ref{thm1} implies that the average number of points on a smooth plane curve is $q+1$, but this is not true anymore if one looks at all plane curves. Our heuristic above shows that for a random polynomial $F \in S_d$ the probability that a point $P\in \mathbb P^2(\mathbb F_q)$ actually lies on $C_F$ is $1/q$. This can also be proven easily (see Section \ref{sectionlowdegree}
for the proof); we record the result here.

\begin{proposition} \label{nonsmooththm}
Let $Y_1, \dots, Y_{q^2+q+1}$ be i.i.d. random variables taking the value $1$ with probability
$1/q$ and the value $0$ with probability $(q-1)/q$.
Then, for $d > q^2+q+1$,
$$\frac{\# \{ F \in S_d ; \#C_F(\F_q))=t \}}{\# S_d} =
{\rm Prob}\left( Y_1 + \dots + Y_{q^2+q+1} = t \right).$$
\end{proposition}

Proposition \ref{nonsmooththm} is an exact result (without an error term)
as there is no sieving involved. For smooth curves, one has to sieve over primes
of arbitrarily large degree, since a smooth curve is required not to have any singular points over  
$\overline{\F}_q,$ not only over $\F_q$. This introduces the error term. In particular, the  average number of points on a plane curve $F(X,Y,Z)=0$ without any smoothness
condition is $q+1 + \frac{1}{q}$.

\subsection*{Some related work}
Brock and Granville \cite{bg} calculated the average number of points in families of curves of given genus $g$ over finite fields,  
\[N_r(g,q)=\sum_{C/{\mathbb F}_q,{\rm genus}(C)=g}{N_r(C)\over |{\Aut}(C/{\mathbb F}_q)|}\Biggl/\sum_{C/{\mathbb F}_q,{\rm genus}(C)=g}{1\over |{\Aut}(C/{\mathbb F}_q)|}\] 
where $N_r(C)$ denotes the number of ${\mathbb F}_{q^r}$-rational points of $C$. It turns out that, depending on the value of $r$, $N_r(g,q)$ shows very different behavior as $q\rightarrow\infty$. Indeed, $N_r(g,q)=q^r+o(q^{r/2})$ unless $r$ is even and $r\leq 2g$, in which case $N_r(g,q)=q^r+q^{r/2}+o(q^{r/2})$. This ``excess''  phenomenon has a natural explanation in terms of Deligne's equidistribution theorem for Frobenius conjugacy classes of the $\ell$-adic sheaf naturally attached to this family, as pointed out by Katz \cite{k}. Using Deligne's theorem, Katz showed that as $q\rightarrow\infty$, $N_r(g,q)$ can be expressed in terms of the integral $I_r(G)=\int_{G}{\rm tr}(A^r)dA$, where $G\ (={\rm USp}(2g)$ in this case) is a compact form of the geometric monodromy group of that sheaf; the occurrence of the excess phenomenon depends on the values of $I_r(G)$, which are computed using the representation theory of $G$. This approach, which is described in a more general form in \cite{ks}, has the advantage of being applicable in other situations in which the geometric monodromy group has been identified, for instance, when calculating the average number of points in the family of smooth degree $d$ hypersurfaces in ${\mathbb P}^n$ over finite fields. In particular, for $n=2$, one obtains the average number of points of smooth planes curves of degree $d$, which are the subject of the present investigation, but from a different point of view.

Namely, while both \cite{bg} and \cite{k} are concerned with curves of fixed genus as the number of points in the base field varies, we consider the complementary situation of working over a fixed field and allowing the genus to vary. We also consider the question of the double limit as both the genus and the number of points in the base field grow. Similar questions were investigated in
\cite{KuRu} for hyperelliptic curves, and in \cite{BuDaFeLa1,
BuDaFeLa2, Xi} for cyclic trigonal curves and general cyclic
$p$-covers.

\section{Poonen's sieve}\label{sieve}

We will adapt the results in Section 2 of \cite{p} to our case,
which is simpler as we take $n=2$ and $X=\mathbb P^2 \subset \mathbb P^2$. But, unlike Poonen, we need to keep track of the error terms.

First let us do this in general in his setup, namely take $Z\subset X$ a finite subscheme. Then $U = \mathbb P^2 \setminus Z$ will automatically be smooth and of dimension $2.$  We will need to choose $Z$ and $T\subset H^0(Z, \mathcal O_Z) $ in a way that imposes the appropriate local conditions for our curves at finitely many points.

The strategy is to check the smoothness separately at points of low, medium, and high degree, and then combine the conditions at the end. The main term will come from imposing conditions on the values taken by both a random polynomial $F \in S_d$ and its first order derivatives at the points in $U$ of relatively small degree (for $d$ large enough). The error term will
come from  smoothness conditions at primes $P$ of large
degree (compared to $d$).

Following  Poonen, denote $A=\mathbb F_q[x_1,x_2]$ and $A_{\leq d}$ the set of polynomials in $A$ of degree at most $d$.
Denote by $U_{<r}$ the closed points of degree $<r$ and by $U_{>r}$ the closed points of $U$ of degree $>r.$ Set
\begin{eqnarray*} \mathcal P_{d,r} &=& \{ F\in S_d; C_F \cap U \textrm{ is smooth of dimension $1$ at all } P \in U_{<r}, F|_Z\in T \},
\\
 \mathcal Q _{r,d} &=& \{ F\in S_d; \exists P \in U \textrm{ s.t. }r\leq \deg P \leq d/3, C_F \cap U \textrm{ is not smooth of dimension $1$ at } P \},
 \\
\mathcal Q ^{high}_{d} &=& \{ F\in S_d; \exists P \in U_{>d/3} \textrm{ s.t. } C_F \cap U \textrm{ is not smooth of dimension $1$ at } P \}.
\end{eqnarray*}

\subsection{Points of low degree}
\label{sectionlowdegree}
All the results of this section depend on the
following lemma proven in \cite{p} using classical results from algebraic geometry.

\begin{lemma}  For any subscheme $Y\subset \mathbb P^2$, the map $\phi_d:S_d =H^0(\mathbb P^2, \mathcal O_{\mathbb P^2}(d)) \rightarrow  H^0(Y, \mathcal O_Y(d))$ is surjective for $d\geq \dim H^0(Y, \mathcal O_Y) - 1.$\end{lemma}

\begin{proof}Take $n=2$ in Lemma 2.1 in \cite{p}.\end{proof}

\begin{lemma} \label{lsmdeg}Let $U_{<r}=\{ P_1, \ldots, P_s\}$. Then for $d \geq 3rs+\dim  H^0(Z,\mathcal O_Z) -1$, we have
\[ \frac{\# \mathcal P_{d,r} } {\# S_d} = \frac{\# T} {\# H^0(Z,\mathcal O_Z)} \prod_{i=1}^s (1-q^{-3\deg P_i}).\]
\end{lemma}

\begin{proof}The result follows from Lemma 2.2 in \cite{p}, as long as we ensure that $d+1$ is bigger than the dimension of $H^0(Z,\mathcal O_Z)\times  \prod_{i=1}^s H^0(Y_i,\mathcal O_{Y_i})$, where $Y_i$ is the closed subscheme corresponding to $P_i$ in the manner described by Poonen. Thus $\dim  H^0(Y_i,\mathcal O_{Y_i}) = 3\deg P_i <3r.$ \end{proof}

\subsection*{Proof of Proposition\ref{nonsmooththm}} We can use this last result to compute the average number of points on the curves $C_F$ associated to
the polynomials $F \in S_d$ without any smoothness condition.  We
pick $P_1, \dots, P_{q^2+q+1}$ an enumeration of the points of
$\P^2(\F_q)$, and we take $Z$ to be a ${\mathfrak m}_P$-neighborhood
for each point $P \in \mathbb P^2(\mathbb F_q)$ (this means that we
look at the value of $F$ at that point; for smoothness, we will also
look at the value of its first order derivatives). Thus
\[ H^0(Z, \mathcal O_Z)= \prod_{P \in \mathbb P^2(\mathbb F_q)} \mathcal O_P/\mathfrak m_P.\]
Each space has dimension $1$, so $\dim H^0(Z, \mathcal O_Z) = q^2+q+1,$ and $\#  H^0(Z, \mathcal O_Z) = q^{q^2+q+1}.$ Let $0 \leq t \leq q^2+q+1$.
We want to count all curves $C_F$ such that $P_1, \dots, P_t \in C_F(\F_q)$ and
$P_{t+1}, \dots, P_{q^2+q+1} \not\in C_F(\F_q)$. We then choose
\[T= \{ (a_i)_{1\leq i \leq q^2+q+1}; a_1, \dots, a_t = 0, a_{t+1}, \dots, a_{q^2+q+1} \in \F_q^\times\},\]
and $|T| =  (q-1)^{q^2+q+1-t}.$ It follows by taking $r=0$ in  Lemma
\ref{lsmdeg} that, when $d \geq q^2+q$,
\begin{eqnarray*}
&&\frac{\# \left\{ F \in S_d ; P_1, \dots, P_t \in C_F(\F_q), P_{t+1}, \dots, P_{q^2+q+1}
\not\in C_F(\F_q) \right\}}{\# S_d} \\ &&= \frac{\# \mathcal P_{d,0} } {\# S_d}
= \frac{\# T} {\# H^0(Z,\mathcal O_Z)}
= \frac{(q-1)^{q^2+q+1-t}}{q^{q^2+q+1}} = \left( \frac{1}{q} \right)^{t}
\left( \frac{q-1}{q} \right)^{q^2+q+1-t}.
\end{eqnarray*}
Then,
\begin{eqnarray*}
 {\rm Prob}\left( \# C_F(\F_q)=t \right) &=& \sum_{{\varepsilon_1, \dots, \varepsilon_{q^2+q+1}
\in \{ 0, 1 \}}\atop{\varepsilon_1 + \dots + \varepsilon_{q^2+q+1}= t}}
 \left( \frac{1}{q} \right)^{t}
\left( \frac{q-1}{q} \right)^{q^2+q+1-t} \\ &=&
{\rm Prob}\left( Y_1 + \dots + Y_{q^2+q+1}=t \right),
\end{eqnarray*}
and this proves Proposition \ref{nonsmooththm}. \qed

Now we want to use Lemma \ref{lsmdeg} to sieve out non-smooth curves. We remark
that Lemma \ref{lsmdeg} gives an exact formula without error term, but we need to choose $r$ as a function of
 $d$ and the product will contribute to the error term. In addition, $s$ itself depends on $r$.

As the number of closed points of degree $e$ in $U$ is bounded by the
number of closed points of degree
 $e$ in $\mathbb P^2$, which is $q^{2e}+q^{e}+1<2q^{2e}$, the product
\[\prod_{P \textrm { closed point of }U}  (1-q^{-z\deg P})^{-1} = \zeta_U (z)\]
converges for $\Re(z)>2.$ For the same reason, we get that
 \begin{equation}\label{zeta}\prod_{i=1}^s (1-q^{-3\deg P_i}) =
 \zeta_U(3)^{-1}\left( 1+ O\left( \frac{q^{-r}}{1-q^{-1}-2q^{-r}} \right)\right).\end{equation}
Indeed, in order to show \eqref{zeta}, we write
\[\prod_{i=1}^s (1-q^{-3\deg P_i}) = \zeta_U(3)^{-1} \prod_{\deg P \geq r} \left(1-q^{-3\deg P}\right)^{-1}. \]
For any sequence of nonnegative numbers $\{x_i\}_i$, we know that
\[1\leq \prod_{i=1}^\infty (1-x_i)^{-1} \leq \frac{1}{1 - \sum x_i}.\]
Taking the sequence in question to be $\{q^{-3\deg P}\}_{\deg P \geq r} $, it means that we need an upper bound for \[\sum_{\deg P\geq r} q^{-3\deg P} = \sum_{j=r}^\infty q^{-3j} \#\{\textrm{closed points of $U$ of degree } j \}.\] All the $P$'s until now have been closed points of $U$, but $U$ is a subset of $\mathbb P^2$, so it has at most $\#\mathbb P^2(F_{q^j})= q^{2j} + q^{j} +1 \leq 2 q^{2j}$ closed points of degree $j$. Hence
\[ \sum_{\deg P\geq r} q^{-3\deg P} \leq 2 \sum_{j=r}^\infty q^{-j} = \frac{2q^{-r}}{1-q^{-1}},\]
and now we get
\[1\leq \prod_{\deg P \geq r} \left(1-q^{-3\deg P}\right)^{-1} \leq \frac{1}{1- \frac{2q^{-r}}{1-q^{-1}}} ,\]
which proves \eqref{zeta}. Substituting \eqref{zeta} in Lemma~\ref{lsmdeg}, we obtain
\begin{equation}\label{smdeg}
\frac{\# \mathcal P_{d,r} } {\# S_d} = \zeta_U(3)^{-1}\frac{\# T} {\# H^0(Z,\mathcal O_Z)} \left( 1+ O\left( \frac{q^{-r}}{1-q^{-1}-2q^{-r}} \right)\right).
\end{equation}

\draftcom{This is a very crude estimate, but maybe we can get away with it.}
\subsection{Points of medium degree}

\begin{lemma} \label{precedent} For a closed point $P \in U$ of degree $e\leq d/3$, we have
\[\frac{ \#\{F\in S_d; C_F \cap U \mbox{ is not smooth of dimension 1 at } P\}}{ \#S_d} = q^{-3e}.\]
\end{lemma}

\begin{proof} Take $m=2$ in Lemma 2.3 of \cite{p}. This also
follows from Lemma \ref{lsmdeg} by taking $r=0$ and $Z$ to be a
${\mathfrak m}_P^2$-neighborhood of $P$ (which mean that we look at
$F$ and its first order derivatives). Then, \[ H^0(Z, \mathcal O_Z)=
 \mathcal O_P/\mathfrak m_P^2,\]
and $\dim H^0(Z, \mathcal O_Z) = q^{3 \deg{P}}$. We also choose $T=
\{ (0,0,0) \},$ as we want $F$ and its first order derivatives to
vanish at $P$.  Then,
\begin{eqnarray*}
\frac{ \#\{F\in S_d; C_F \cap U \mbox{ is not smooth of dimension 1
at } P\}}{ \#S_d}  &=&
\frac{\# T} {\# H^0(Z,\mathcal O_Z)} = q^{-3 \deg{P}}.
\end{eqnarray*}
\end{proof}

\begin{lemma}
\[ \frac{\# \mathcal Q_{r,d}} {\# S_d} \leq 2 \frac{q^{-r}} {1-q^{-1}}.\]
\end{lemma}

\begin{proof}
We follow the proof of Lemma 2.4 of \cite{p}. We have that \begin{eqnarray*}
 \frac{\# \mathcal Q_{r,d}} {\# S_d} \leq \sum_{{P \in U} \atop {\deg{P}=r}}^{d/3}
\frac{\# \{ F\in S_d; C_F \cap U \textrm{ is not smooth of dimension
$1$ at } P \}}{\# S_d},
\end{eqnarray*}
and $\#U(\mathbb F_{q^e}) \leq \#\mathbb P^2(\mathbb  F_{q^e}) =
q^{2e} + q^{e}+1\leq 2 q^{2e}.$ Then, using Lemma \ref{precedent},
we have
\begin{eqnarray*}
 \frac{\# \mathcal Q_{r,d}} {\# S_d} \leq 2 \sum_{e=r}^{d/3}
 q^{-e} \leq 2 \sum_{e=r}^{\infty} q^{-e} = 2 \frac{q^{-r}} {1-q^{-1}}.
\end{eqnarray*}

\end{proof}

\subsection{Points of high degree}

\begin{lemma} For $P \in \mathbb A^2(\mathbb F_q)$ of degree $e$, we have
\[\frac{ \#\{f\in A_{\leq d}; f(P) =0 \}}{ \# A_{\leq d}} \leq  q^{-\min(d+1,e)} .\]\
\end{lemma}
\begin{proof} Take $n=2$ in Lemma 2.5 of \cite{p}. \end{proof}

\begin{lemma}
\[\frac{\#\mathcal Q^{high}_d} {\# S_d} \leq 3 (d-1)^2 q^{-\min\left(\left\lfloor \frac{d}{p}\right\rfloor +1, \frac{d}{3} \right)} + 3d q^{-\lfloor\frac{d-1}{p}\rfloor-1}. \]
\end{lemma}

\begin{proof} If we get a bound for all affine $U\subset \mathbb A^2$, the same bound multiplied by $3$ will hold for $U\subset \mathbb P^2$, since it can be covered by three affine charts. So we can reduce the problem to affine sets.
We follow the proof of Lemma 2.6 of \cite{p}, while keeping track of
the constants appearing in the error terms, which is not done in
\cite{p} as only the main term is needed for his application. In our
case the coordinates are simply $x_1$ and $x_2$, which have degree
$1$ and $D_i=\frac{\partial}{\partial x_i}$, $i=1,2$ are already
global derivations. This allows us to work globally on the set $U$
and there is no need to work locally as in \cite{p}. Now we can work
with dehomogenizations of polynomials in $S_d$, so we need to find
the polynomials $f \in A_{\leq d}$ for which $C_f \cap U$ fails to
be smooth at some $P \in U$. This happens if and only if
$f(P)=(D_1f)(P) =(D_2f)(P)=0.$

 Any polynomial $f \in A_{\leq d}$ can be written as
\[ f = g_0 + g_1^p x_1 + g_2^p x_2 + h^p\]
with $g_0 \in A_{\leq d}$, $g_1, g_2 \in A_{\leq \gamma}$ and $h \in A_{\leq \eta},$ where $\gamma = \left\lfloor\frac{d-1}{p}\right\rfloor$ and $\eta =\left\lfloor \frac{d}{p}\right\rfloor.$ The ``trick'' used by Poonen is based on the observation that selecting $f$ uniformly at random amounts to selecting $g_0, g_1, g_2$ and $h$ independently and uniformly at random. The advantage is that $D_i f = D_i g_0 + g_i^p,$ so each derivative depends only on $g_0$ and one of the $g_1, g_2$.  Set
\[W_0 = U, \quad W_1= U\cap \{D_1f =0\}, \quad W_2= U\cap \{D_1f = D_2f=0\}.\]

\emph{Claim 1.}  For $i=0,1$ and for each choice of $g_0, \ldots,  g_i$, such that $\dim W_i \leq 2-i$,

\[ \frac{\#\{(g_{i+1},\ldots, g_2, h); \dim W_{i+1} \leq  1-i  \} }{\#\{(g_{i+1},\ldots, g_2, h)  \} } \leq (d-1)^i q^{-\lfloor\frac{d-1}{p}\rfloor-1} .\]

B\'ezout's theorem tells us that the number of $(2-i)$-dimensional components of $(W_i)_{\textbf{red}}$ is bounded above by $(d-1)^i, $ since $\deg D_if \leq d-1,$ for each $i,$ and $\deg \overline{ U} =1.$ The rest of the argument follows from Poonen's computation.

\emph{Claim 2.} For any choice of $g_0,\ldots, g_2$,
\[\frac{\#\{h; C_f\cap W_2 \cap U_{>d/3}\}}{\{ \textrm {all } h\}} \leq (d-1)^2 q^{-\min\left(\left\lfloor \frac{d}{p}\right\rfloor+1, \frac{d}{3} \right)}. \]

This follows from the fact that $\# W_2 \leq (d-1)^2$ (from the same B\'ezout argument as before) and the coset argument in \cite{p} shows that for each $P \in W_2,$ the set of bad $h$'s at $P$ is either empty or has density at most  $q^{-\min\left(\left\lfloor \frac{d}{p}\right\rfloor+1, \frac{d}{3} \right)}$ in the set of all $h$ (because $\deg P>d/3$).

To finish the proof of the lemma, put the two claims together and get that

\[\frac{\#\mathcal Q^{high}_d} {\# S_d}
\leq 3 (d-1)^2 q^{-\min\left(\left\lfloor \frac{d}{p}\right\rfloor +1, \frac{d}{3} \right)} + 3d q^{-\lfloor\frac{d-1}{p}\rfloor-1}.\]

\end{proof}


Combining the points of small, medium and high degree, we get that
\begin{eqnarray} \label{countsmooth}
&&\{F\in S_d; C_F \cap U \textrm{ is smooth of dimension $1$ and
}F|_Z\in T  \} \\ \nonumber &&= \frac{\# T}{\zeta_U(3) \#
H^0(Z,\mathcal O_Z)} \left( 1 + O \left( \frac{q^{-r}}{1 - q^{-1} -
2q^{-r}} \right) \right) \\ &&  \nonumber + O \left(
\frac{q^{-r}}{1-q^{-1}} + (d-1)^2 q^{-\min\left(\left\lfloor
\frac{d}{p}\right\rfloor +1, \frac{d}{3} \right)} + d
q^{-\lfloor\frac{d-1}{p}\rfloor-1} \right)
\end{eqnarray}

We need to choose an appropriate value for $r$. According to Lemma
\ref{lsmdeg}, we must have  $d \geq 3rs + \dim H^0(Z, \mathcal O_Z)
-1$ and $\frac{1}{r}(q^{2r} +q^r+1)<s<q^{2r}+q^r+1$. When using
\eqref{countsmooth} in Section \ref{numberpoints}, we will always have $Z
\subset \mathbb{P}^2(\F_q),$ thus $\dim H^0(Z, \mathcal
O_Z)<6q^2$. We take $r=\frac{3B+\log d}{3}$ for any fixed constant
$B\geq 0$. With this choice of $r$, the error term of
\eqref{countsmooth} coming from points of medium and high degree is
therefore
\begin{equation}\label{et}O\left(\frac{q^{-B}d^{-1/3}}{1-q^{-1}}+
(d-1)^2 q^{-\min\left(\left\lfloor \frac{d}{p}\right\rfloor +1,
\frac{d}{3} \right)} + d
q^{-\lfloor\frac{d-1}{p}\rfloor-1}\right).\end{equation}



\section{Number of points}
\label{numberpoints}

We apply the results in Section~\ref{sieve} twice. The first time to
evaluate the fraction of homogeneous polynomials of degree $d$ that
define smooth plane curves, and the second time to evaluate the
fraction of homogeneous polynomials of degree $d$ that define smooth
plane curves with predetermined $\mathbb F_q$-points. By taking the
quotient we then obtain an asymptotic formula for the fraction of
smooth plane curves that have predetermined $\mathbb F_q$-points.

For the first evaluation, we take $Z=\emptyset$ and $T=\{0\}$ in
\eqref{countsmooth} to get
\begin{align}\label{nopts}
\frac {\#\{F\in S_d^{\rm ns} \}}{\#S_d} =&  \zeta_{{\mathbb P}^2}(3)^{-1}\left( 1+ O\left( \frac{q^{-B}d^{-1/3}}{1-q^{-1}-2q^{-B}d^{-1/3}} \right) \right)\\ \nonumber
&+ O\left(\frac{q^{-B}d^{-1/3}}{1-q^{-1}}+ (d-1)^2 q^{-\min\left(\left\lfloor \frac{d}{p}\right\rfloor+1, \frac{d}{3} \right)} + d q^{-\lfloor\frac{d-1}{p}\rfloor-1}\right).
\end{align}

Pick $P_1, \ldots P_{q^2+q+1}$ an enumeration of the points of $\mathbb P^2(\mathbb F_q)$,
and let $0 \leq t \leq q^2+q+1$.  We want to compute
\begin{eqnarray*}\label{pts}\frac{\#\{F\in S_d^{\rm ns} ;  P_1, \dots, P_t \in C_F(\F_q),
P_{t+1}, \dots, P_{q^2+q+1} \not\in
C_F(\F_q)\}}{\#S_d}.\end{eqnarray*}

This is achieved by taking $Z$ to be an $\mathfrak m^2_P$-neighborhood for each point $P \in \mathbb P^2(\mathbb F_q)$
(this means that we look at the value of $F$ and its first order derivatives at each point). Thus
\begin{eqnarray} \label {choiceforZ} H^0(Z, \mathcal O_Z)= \prod_{P \in \mathbb P^2(\mathbb F_q)} \mathcal O_P/\mathfrak m^2_P.\end{eqnarray}
Each space has dimension $3$, so $\dim H^0(Z, \mathcal O_Z) =
3(q^2+q+1),$ and $\#  H^0(Z, \mathcal O_Z) = q^{3(q^2+q+1)}.$

Then we want $T$ to be the set of $((a_i, b_i, c_i))_{1\leq i \leq q^2+q+1}$
such that $a_1, \ldots, a_t = 0$, $a_{t+1}, \dots, a_{q^2+q+1}$ $\in \mathbb F_q^\times$,
and $(a_i, b_i, c_i)\neq (0,0,0)$ for $1 \leq i \leq q^2+q+1$. This gives
$$|T| = (q^2-1)^{t} (q-1)^{q^2+q+1-t} q^{2(q^2+q+1-t)}.$$

Using \eqref{countsmooth} with this choice of $Z$ and $T$,  we
obtain
\begin{eqnarray}\label{pts2}
&&\frac{\#\{F\in S_d^{\rm ns} ;  P_1, \dots, P_t \in C_F(\F_q),
P_{t+1}, \dots, P_{q^2+q+1} \not\in  C_F(\F_q)\}}{\#S_d}\\
\nonumber &&= \zeta_{U}(3)^{-1} \frac{(q^2-1)^{t} (q-1)^{q^2+q+1-t}
q^{2(q^2+q+1-t)}}{q^{3(q^2+q+1)}} \left( 1+ O\left(
\frac{q^{-B}d^{-1/3}}{1-q^{-1}-2q^{-B}d^{-1/3}} \right) \right)\\
\nonumber &&+ O\left(\frac{q^{-B}d^{-1/3}}{1-q^{-1}}+ (d-1)^2
q^{-\min\left(\left\lfloor \frac{d}{p}\right\rfloor+1, \frac{d}{3}
\right)} + d q^{-\lfloor\frac{d-1}{p}\rfloor -1}\right).
\end{eqnarray}


Here, by multiplicativity of zeta functions,
$$ \frac{\zeta_{{\mathbb P}^2}(z)}{\zeta_{U}(z)} = \zeta_{Z}(z) = \left( \frac{1}{1-q^{-z}} \right)^{q^2+q+1}
\;\;\;\mbox{for $U = {\mathbb P}^2 \setminus Z.$}$$
Then by taking the quotient of \eqref{pts2} and \eqref{nopts}, we get that
\begin{align*} &\frac{\#\{F\in S_d^{\rm ns} ;  P_1, \dots, P_t \in C_F(\F_q),
P_{t+1}, \dots, P_{q^2+q+1} \not\in  C_F(\F_q)\}}
{\# S_d^{\rm ns}} \\
&=  \left( \frac{q^3}{q^3-1} \right)^{q^2+q+1} \frac{(q^2-1)^{t} (q-1)^{q^2+q+1-t} q^{2(q^2+q+1-t)}}{q^{3(q^2+q+1)}}\\
&\times \left(1+ O\left(q^{t}\left({q^{-B}d^{-1/3}}+ + (d-1)^2 q^{-\min\left(\left\lfloor \frac{d}{p}\right\rfloor+1, \frac{d}{3} \right)} + d q^{-\lfloor\frac{d-1}{p}\rfloor -1}\right)\right) \right)\\
&=
\left( \frac{q+1}{q^2+q+1} \right)^{t} \left( \frac{q^2}{q^2+q+1} \right)^{q^2+q+1-t} \\&\times\left(1+ O\left(q^{t}\left({q^{-B}d^{-1/3}}+ + (d-1)^2 q^{-\min\left(\left\lfloor \frac{d}{p}\right\rfloor+1, \frac{d}{3} \right)} + d q^{-\lfloor\frac{d-1}{p}\rfloor -1}\right)\right) \right).
 \end{align*}

Theorem \ref{thm1} follows by taking $B=0$ above and noting that for any $\varepsilon_1, \dots, \varepsilon_{q^2+q+1} \in \left\{ 0, 1 \right\}$ with $\varepsilon_1 + \dots +  \varepsilon_{q^2+q+1}=t$, 
\[\Prob \left( X_1 = \varepsilon_1,
\dots, X_{q^2+q+1} = \varepsilon_{q^2+q+1} \right) = \left( \frac{q+1}{q^2+q+1} \right)^{t} \left( \frac{q^2}{(q^2+q+1)} \right)^{q^2+q+1-t} .\]

\section{Moments}

By Theorem \ref{thm1} the number of points of smooth plane curves over $\F_q$ is distributed as $X_1 + \dots + X_{q^2+q+1}$,
and the trace of the Frobenius as $X_1 + \dots + X_{q^2+q+1} - (q+1)$.(The mean of $X_1, \dots, X_{q^2+q+1}$ is $q+1$.)
Applying the triangular central limit theorem to the random variables $X_1, \dots, X_{q^2+q+1}$,
we have that
$(X_1 + \dots + X_{q^2+q+1} - (q+1)) / \sqrt{q+1}$ is distributed as $N(0,1)$
when $q \rightarrow \infty$.

We would like to say the same thing about the distribution of the trace of Frobenius in our family when
$d$ and $q$ go to infinity, which amounts to the computation of the moments.

We will first compute
\[N_k(q,d)=\frac{1}{\# S_d^{\rm ns}} \sum_{F\in S_d^{\rm ns}}  \left(\frac{\#C_F(\mathbb{F}_q)}{\sqrt{q+1}} \right)^k,\]
and then deduce the result for $M_k(q,d)$.

By using an exponential sum to count the number of points in the curve, we can write
\begin{align*}
N_k(q,d)
=&\frac{1}{\# S_d^{\rm ns}} \left(\frac{1}{\sqrt{q+1}} \right)^k \sum_{F\in S_d^{\rm ns}}  \left(\sum_{P\in \mathbb{P}^2(\mathbb{F}_q) }S_F(P)\right)^k,\\
\end{align*}
where
\[S_F(P)=\frac{1}{q}\sum_{t \in \mathbb{F}_q}e\left( \frac{tF(P)}{q}\right).\]
Thus,  expanding the $k$-th power,
\begin{eqnarray*}
N_k(q,d)&=&\frac{1}{\# S_d^{\rm ns}}
\left(\frac{1}{\sqrt{q+1}} \right)^k \sum_{P_1, \dots, P_k \in \mathbb{P}^2(\mathbb{F}_q)}
\sum_{F \in S_d^{\rm ns}} S_F(P_1) \dots S_F(P_k) \\
&=&\frac{1}{\# S_d^{\rm ns}}
\frac{1}{(q+1)^{k/2}} \sum_{\ell=1}^{\min(k, q^2+q+1)}  h(\ell,k)\sum_{({\bf P},{\bf b})\in P_{\ell,k}}
\sum_{F \in S_d^{\rm ns}} S_F(P_1)^{b_1} \dots S_F(P_\ell)^{b_\ell},
\end{eqnarray*}
where
\begin{eqnarray*}
P_{\ell,k}= &&\left\{ ({\bf P},{\bf b}); {\bf P}=(P_1, \dots,P_\ell) \;\mbox{with $P_i$ distinct points of  $\mathbb{P}^2(\mathbb{F}_q)$},  \right.\\
&&\left. {\bf b} =(b_1,\dots,b_\ell) \;\mbox{with $b_i$ positive integers such that
$b_1 + \dots +b_\ell = k$} \right\}.\end{eqnarray*}
Notice that
\[\sum_{\ell=1}^k  h(\ell,k)\sum_{({\bf P},{\bf b})\in P_{\ell,k}}1=(q^2+q+1)^k.\]

Now let us fix $({\bf P}, {\bf b}) \in P_{\ell,k}$. Then,
\begin{align*}
&\frac{1}{\# S_d^{\rm ns}} \sum_{F\in S_d^{\rm ns}}
S_F(P_1)^{b_1} \dots S_F(P_\ell)^{b_\ell}=\sum_{a_1, \dots, a_\ell \in \mathbb{F}_q} \frac{1}{\# S_d^{\rm ns}}
\sum_{{ F \in S_d^{\rm ns}}\atop{F(P_j)=a_j}} \prod_{j=1}^\ell S(a_j)^{b_j},
 \end{align*}
where \[S(a) = \frac{1}{q}\sum_{t \in \mathbb{F}_q}e\left( \frac{t a}{q}\right)=\left\{\begin{array}{cc}1 & a=0,\\ 0 & a \not = 0.\end{array}\right.\]
The $b_j$ have no influence on the result and we obtain nonzero terms only when $a_j=0$ for $1\leq j\leq \ell$, and in this case
\begin{eqnarray}
\frac{1}{\# S_d^{\rm ns}} \sum_{F\in S_d^{\rm ns}}
S_F(P_1)^{b_1} \dots S_F(P_\ell)^{b_\ell} &=&
\frac{ \# \left\{  F \in S_d^{ns} ; F(P_i)=0 \;\mbox{for $1 \leq i \leq \ell$} \right\} }
{\# S_d^{\rm ns}}.
\end{eqnarray}

We want to point out that this can also be computed similarly to the work in Section \ref{numberpoints}. 

Choose the constant $B=k$, $Z$ as in
\eqref{choiceforZ} and $T$ to be the set of $\left((a_i, b_i, c_i)\right)_{1 \leq i \leq q^2+q+1}$
such that $a_1 = \dots = a_\ell=0$, $a_{\ell+1}, \dots, a_{q^2+q+1} \in \F_q$
and $(a_i, b_i, c_i) \neq (0,0,0)$ for $1 \leq i \leq q^2+q+1$. Then,
$$|T| = (q^2-1)^{\ell} (q^3-1)^{q^2+q+1-\ell},$$
and
\begin{eqnarray*}
&&\frac{ \# \left\{  F \in S_d^{ns} ; F(P_i)=0 \;\mbox{for $1 \leq i \leq \ell$} \right\} }
{\# S_d^{\rm ns}} \\&=&
\left(\frac{q+1}{q^2+q+1}\right)^\ell \left(1+ O\left({q^{-k}d^{-1/3}q^{\ell}}+ (d-1)^2 q^{\ell-\min\left(\left\lfloor \frac{d}{p}\right\rfloor+1, \frac{d}{3} \right)} + d q^{\ell-\lfloor\frac{d-1}{p}\rfloor-1}\right)\right) .
\end{eqnarray*}

Now we sum over all the elements in $P_{\ell,k}$:
\begin{eqnarray*}
N_k(q,d) &=& \frac{1}{(q+1)^{k/2}}  \sum_{\ell=1}^{\min(k, q^2+q+1)}  h(\ell,k)\sum_{({\bf P},{\bf b})\in P_{\ell,k}} \left(\frac{q+1}{q^2+q+1}\right)^\ell
\\
&\times & \left(1+ O\left( q^{\min(k, q^2+q+1)} \left( q^{-k}d^{-1/3} + (d-1)^2 q^{-\min\left(\left\lfloor \frac{d}{p}\right\rfloor+1, \frac{d}{3} \right)} + d q^{-\lfloor\frac{d-1}{p}\rfloor-1}\right)\right) \right) .
\end{eqnarray*}

On the other hand, we have
\[\mathbb{E}\left(\left(\frac{1}{\sqrt{q+1}} \sum_{i=1}^{q^2+q+1} X_i \right)^k\right)= \left(\frac{1}{\sqrt{q+1}}\right)^k \sum_{\ell=1}^k h(\ell, k) \sum_{({\bf i,b})\in A_{\ell,k}}\mathbb{E}\left(X_{i_1}^{b_1}\dots X_{i_\ell}^{b_\ell}\right) , \]
where
\[A_{\ell,k}=\left\{({\bf i,b}); {\bf i}=(i_1,\dots,i_\ell), 1\leq i_j \leq q^2+q+1 \textrm{ distinct }, {\bf b}=(b_1,\dots, b_\ell) \sum_{j=1}^\ell b_j = k \right\}.\]
Since
\[\mathbb E(X_1^{b_1} \dots X_\ell^{b_\ell} ) = \left( \frac{q+1}{q^2+q+1} \right)^\ell \]
and $\# P_{\ell, k} = \# A_{\ell, k}$, we conclude that
\begin{eqnarray} \label{MT}
&&N_k(q,d) = \mathbb{E}\left(\left(\frac{1}{\sqrt{q+1}} \sum_{i=1}^{q^2+q+1} X_i \right)^k\right)\\
\label{ET}
&\times& \left(1+ O\left( q^{\min(k, q^2+q+1)} \left(q^{-k} d^{-1/3} + (d-1)^2 q^{-\min\left(\left\lfloor \frac{d}{p}\right\rfloor+1, \frac{d}{3} \right)} + d q^{-\lfloor\frac{d-1}{p}\rfloor-1}\right)\right) \right) .
\end{eqnarray}

Now using \eqref{MT} and the binomial theorem, we get that
\begin{align*}
M_k(q,d) =& \sum_{j=0}^k \binom{k}{j} N_j(q,d) \left(-\sqrt{q+1}\right)^{k-j}\\
\sim&\sum_{j=0}^k \binom{k}{j} \mathbb{E}\left(\left(\frac{1}{\sqrt{q+1}} \sum_{i=1}^{q^2+q+1} X_i \right)^j\right)\left(-\sqrt{q+1}\right)^{k-j}\\
=& {\mathbb{E}} \left( \left( \frac{1}{\sqrt{q+1}} \left( \sum_{i=1}^{q^2+q+1} X_i - (q+1) \right) \right)^k \right)
\end{align*}
with the same error term as \eqref{ET}.
This completes the proof of Theorem \ref{thm2}.

\noindent {\bf Acknowledgments}. 
The authors wish to thank Pierre Deligne and Ze\'{e}v Rudnick for suggesting the problem that we consider
in this paper. The authors are grateful to both of them as well as P\"ar Kurlberg for helpful discussions.
This work was supported by the Natural Sciences and Engineering Research Council
of Canada [B.F., Discovery Grant 155635-2008 to C.D., 355412-2008 to M.L.] and
the National Science Foundation of U.S. [DMS-0652529 and DMS-0635607 to A.B.]. M.L. is also supported by a Faculty of Science Startup grant from the University of Alberta, and C.D. is also supported by a grant to the Institute for Advanced Study from the Minerva Research Foundation.

\end{document}